\documentclass[12pt]{amsart}
\usepackage{amssymb,amsmath,amsthm,andy,fullpage}
\usepackage{pdfsync}
\usepackage{color}
\newtheorem{thm}{Theorem}[section]

\newtheorem{lem}[thm]{Lemma}
\newtheorem{cor}[thm]{Corollary}

\theoremstyle{definition}
\newtheorem{defn}[thm]{Definition}

\theoremstyle{remark}
\newtheorem{rem}[thm]{Remark}

\newcommand{\abs}[1]{\left\vert#1\right\vert}
\newcommand{\pnorm}[1]{\left\Vert#1\right\Vert}
\newcommand{\set}[1]{\left\{#1\right\}}

\newcommand{\eps}{\varepsilon}

\DeclareMathOperator{\range}{Range}
\DeclareMathOperator{\dist}{dist}
\DeclareMathOperator{\re}{Re}
\DeclareMathOperator{\im}{Im}
\renewcommand{\P}{\mathbb {P}}

\DeclareMathOperator{\rea}{Reach}
\DeclareMathOperator{\unp}{Unp}

\begin{document}

\title{Defining Functions for Unbounded $C^m$ Domains}

\author{Phillip S. Harrington and Andrew Raich}%
\address{Department of Mathematical Sciences \\ 1 University of Arkansas \\ SCEN 301 \\ Fayetteville, AR 72701}
\email{psharrin@uark.edu, araich@uark.edu}

\thanks{The first author is partially supported by NSF grant DMS-1002332 and
second author is partially supported by NSF grant DMS-0855822}

\begin{abstract}For a domain $\Omega\subset\R^n$, we introduce the concept of a uniformly $C^m$ defining function.
We characterize uniformly $C^m$ defining functions in terms of the signed distance function for the boundary and provide a large class of examples of unbounded domains with
uniformly $C^m$ defining functions. Some of our results extend results from the bounded case.
\end{abstract}



\maketitle

\section{Introduction}

Let $\Omega\subset\mathbb{R}^n$ be an open set.  A \emph{$C^m$ defining function}, $m\geq 1$, for $\Omega$ is a real-valued $C^m$ function $\rho$ defined on a neighborhood $U$ of $\partial\Omega$ such that $\set{x\in U:\rho(x)<0}=\Omega\cap U$ and $\nabla\rho\neq 0$ on $\partial\Omega$.  If $\Omega$ has a $C^m$ defining function, we say that $\Omega$ is a \emph{$C^m$ domain}.

For many applications on unbounded domains, the preceding definition is inadequate.  For example, to work in local coordinates that are adapted to the boundary, it is necessary to work in a neighborhood whose size depends on the $C^2$ norm of the defining function.  If the $C^2$ norm is not uniformly bounded, then such neighborhoods may need to be arbitrarily small, which means that a partition of unity subordinate to these neighborhoods might not have uniform bounds on the derivatives.  Other problems might arise in constructions which involve choosing a constant large enough to bound quantities depending on derivatives of the defining function.  Typical results on $C^m$ domains will require the following:
\begin{defn}
  Let $\Omega\subset\mathbb{R}^n$, and let $\rho$ be a $C^m$ defining function for $\Omega$ defined on a neighborhood $U$ of $\partial\Omega$ such that
  \begin{enumerate}
    \item $\dist(\partial\Omega,\partial U)>0$,

    \item $\pnorm{\rho}_{C^m(U)}<\infty$,

    \item $\inf_U|\nabla\rho|>0$.
  \end{enumerate}
  We say that such a defining function is \emph{uniformly $C^m$}.  If $\rho$ on $U$ is uniformly $C^m$ for all $m\in\mathbb{N}$, we say $\rho$ is \emph{uniformly $C^\infty$}.
\end{defn}
On bounded domains, compactness of the boundary implies that every bounded $C^m$ domain has a uniformly $C^m$ defining function.  On unbounded $C^m$ domains with noncompact boundaries, these properties may not hold.  For example, consider $\Omega\subset\mathbb{R}^3$ defined by $\Omega=\set{z<x y^2}$.  This is a $C^{\infty}$ domain, and any $C^2$ defining function $\rho$ for $\Omega$ will take the form $\rho(x,y,z)=h(x,y,z)(z-x y^2)$ for a $C^1$ function $h$ satisfying $h>0$ on $\partial\Omega$.  If we restrict to the line $\ell=\set{y=z=0}\subset\partial\Omega$, we see that $|\nabla\rho||_\ell=h$ and $\frac{\partial^2\rho}{\partial y^2}|_\ell=-2x h$.  If $|\nabla\rho|>C_1>0$ on $U$ then $h>C_1$ on $\ell$, but if $\pnorm{\rho}_{C^2(U)}<C_2$ then $2\abs{x}h<C_2$.  This is impossible if $\abs{x}\geq\frac{C_2}{2 C_1}$, so no defining function for $\Omega$ is uniformly $C^2$, even though the domain itself is $C^\infty$.

The natural choice for a defining function is the signed distance function.  For $\Omega\subset\mathbb{R}^n$ with $C^m$ boundary, define the signed distance function for $\Omega$ by
\[
  \tilde\delta(x)=\begin{cases}d(x,\partial\Omega) & x\notin\Omega\\-d(x,\partial\Omega) & x\in\overline\Omega\end{cases}.
\]
Note that the distance function $\delta(x) := d(x,\partial\Omega)=|\tilde\delta(x)|$ for any $x\in\mathbb{R}^n$.

Let $\unp(\partial\Omega) = \set{ x\in\R^n : \text{ there exists a unique point }y\in \partial\Omega \text{ such that }\delta(x) = |y-x|}$.
The following concepts were introduced in \cite{Fed59}.
\begin{defn}If $y\in \partial\Omega$, then define the \emph{reach} of $\partial\Omega$ at $y$ by
\[
\rea(\partial\Omega,y) = \sup\set{r\geq 0: B(y,r)\subset \unp(\partial\Omega)}
\]
and the reach of $\partial\Omega$ to be
\[
\rea(\partial\Omega) = \inf\set{ \rea(\partial\Omega,y): y\in \partial\Omega}.
\]
\end{defn}

Our main result is the following:
\begin{thm}\label{thm:main thm}
  Let $\Omega\subset\mathbb{R}^n$ be a $C^m$ domain, $m\geq 2$.  Then the following are equivalent:
  \begin{enumerate}
    \item \label{item:uniform} $\Omega$ has a uniformly $C^m$ defining function.

    \item \label{item:delta} $\partial\Omega$ has positive reach, and for any $0<\epsilon<\rea(\partial\Omega)$, the signed distance function satisfies
    $\|\tilde\delta\|_{C^m(U_\ep)}<\infty$ on $U_\epsilon=\set{x\in\mathbb{R}^n:\delta(x)<\epsilon}$.

    \item \label{item:rho} There exists a $C^m$ defining function $\rho$ for $\Omega$ and a constant $C>0$ such that for every point $p\in\partial\Omega$ with local coordinates $\set{y_1,\ldots,y_n}$ satisfying $\frac{\partial\rho}{\partial y_j}(p)=0$ for $1\leq j\leq n-1$, we have
        \[
          \abs{\nabla\rho(p)}^{-1}\abs{\frac{\partial^k\rho(p)}{\partial y_I\partial y_n^j}}<C
        \]
        where $I$ is a multi-index of length $k-j$ with $n\notin I$ for any integers $2\leq k\leq m$ and $0\leq j\leq\min\set{m-k,k}$.
  \end{enumerate}
\end{thm}

\begin{rem}
  An important consequence of this theorem is that our definition of uniformly $C^\infty$ is not too strong.  If for every $m\in\mathbb{N}$ there exists a defining function $\rho_m$ on $U_m$ such that $\rho_m$ is uniformly $C^m$ on $U_m$, then there exists a uniformly $C^\infty$ defining function $\rho$, and we can take $\rho$ to be the signed distance function.
\end{rem}

\begin{rem}\label{rem:Krantz/Parks extension}
In \cite{KrPa81}, Krantz and Parks show that if $\Omega$ is a $C^m$ domain, $m\geq 2$, then there exists a neighborhood $U\supset \partial\Omega$ on which $\tilde\delta$ is $C^m$.
Part \eqref{item:delta} of Theorem \ref{thm:main thm} extends their result by showing that $\tilde\delta$ is $C^m$ up to $\rea(\partial\Omega)$.
\end{rem}

\begin{proof}
  That \eqref{item:delta} implies \eqref{item:uniform} and \eqref{item:uniform} implies \eqref{item:rho} are immediate from the definitions.  That \eqref{item:rho} implies \eqref{item:delta} will follow from Lemmas \ref{lem:rho_delta} and \ref{lem:delta_neighborhood}, proved in Section \ref{sec:basic_results}.
\end{proof}

When studying the asymptotic behavior of a domain, it is natural to consider the domain after embedding $\mathbb{R}^n\subset\mathbb{RP}^n$, and we will do so in Section \ref{sec:projective_space}.  Our theorem will make it easy to check that any $C^m$ domain in $\mathbb{R}^n$ which can be extended to a $C^m$ domain in $\mathbb{RP}^n$ under this embedding will have a uniformly $C^m$ defining function.  However, we will also show that there are examples which are not even $C^1$ in $\mathbb{RP}^n$ but still have uniformly $C^m$ defining functions.

We conclude the paper in Section \ref{sec:applications} with two specific applications of uniformly $C^m$ defining functions. The first is the construction of weighted Sobolev spaces
on unbounded domains, and the second is a brief example from several complex variables to illustrate the advantages of uniformly $C^m$ defining functions in generalizing some well-known constructions.

Over the course of several papers, we will study domains $\Omega$ that admit a uniformly $C^m$ defining function, build weighted Sobolev spaces on them, and develop the elliptic theory associated
to the Sobolev spaces \cite{HaRa13s}. We will then be in a position to investigate the  the $\dbar$-Neumann and $\dbarb$-problems in weighted $L^2$ on $\Omega\subset\C^n$.
Gansberger has obtained compactness results for the $\dbar$-Neumann operator in weighted $L^2$ \cite{Gan10}, but (at the time)
there was neither the elliptic theory nor suitable Sobolev space theory
to study the $\dbar$-Neumann problem in $H^s$ or facilitate the passage from the $\dbar$-Neumann operator at the Sobolev scale $s=1/2$ to the complex Green operator on $\partial\Omega$ in
weighted $L^2$. There are other results about solution operators to $\dbar$ in the unbounded setting but for the case $\Omega = \C^n$, rendering any boundary discussion  moot \cite{HaHe07,Gan11}.

\section{Basic Results}
\label{sec:basic_results}

To  handle rigorously multi-indices with possibly repeated indices, we identify functions with sets of ordered pairs and define a multi-index of length $k$ to be a function
$I:S\rightarrow\set{1,\ldots,n}$ defined on a subset $S$ of the natural numbers such that $\abs{I}=\abs{S}=k$.  If $S=\set{s_1,\ldots,s_k}$ where $\set{s_j}$ is an increasing sequence, we write $I_j=I(s_j)$.  Hence, $\frac{\partial}{\partial x_I}=\frac{\partial}{\partial x_{I_1}}\cdots\frac{\partial}{\partial x_{I_k}}$.  We will identify $I$ with its range, and write $n\in I$ to mean $n\in\range(I)$.  The set of all increasing multi-indices is defined by
\[
  \mathcal{I}_k=\set{I:\set{1,\ldots,k}\rightarrow\set{1,\ldots,n},I\text{ is an increasing function.}}.
\]
By the identification of a function with a set of ordered pairs, all set theoretic operations are defined for multi-indices.

Below, we will take the $C^k$ norm of a function on $\partial\Omega$. We take an extrinsic view, and  for a $C^k$ function $f$ defined on a neighborhood of $\partial\Omega$, we set
\[
\|f\|^2_{C^k(\partial\Omega)} = \sup_{p\in\partial\Omega} \sum_{j=0}^k \sum_{I\in\I_j} \Big| \frac{\p^j f(p)}{\p x_I}\Big|^2=\inf_{U\supset\partial\Omega}\pnorm{f}^2_{C^k(U)}.
\]
The intrinsic $C^k$ norm of a defining function is always zero, hence our use of the extrinsic norm.

For $p\in\partial\Omega$, let $\set{y_1,\ldots,y_n}$ be orthonormal coordinates such that $\nabla\tilde\delta(p)=(0,\ldots,0,1)$.  For functions $f$ defined in a neighborhood of $p$, we define a family of special $C^k$ norms that is adapted to the boundary.  For any integer $k\geq 0$, define
\[
  \abs{f}^2_{C^{k}_b(p)}=\sum_{k'=0}^{k}\sum_{j'=0}^{\min\set{k-k',k'}}\sum_{I\in\mathcal{I}_{k'-j'},n\notin I}\abs{\frac{\partial^{k'} f(p)}{\partial y_I\partial y_n^{j'}}}^2.
\]
The $C^k_b$ norms provide a balance between computability (derivatives are only with respect to $\set{y_j}$) and theoretical elegance (intrinsic tangential derivatives and the normal). In particular,
terms in the $C^k_b$ norm agree with terms in the expansion of a $k$-fold composition of tangential differential operators with respect to local coordinates.
For the purposes of induction, we also define for any integers $k\geq 1$ and $k\geq 2j\geq 0$
\[
  \abs{f}^2_{C^{k,j}_b(p)}=\abs{f}^2_{C^{k-1}_b(p)}+\sum_{j'=j}^{\lfloor k/2\rfloor}\sum_{I\in\mathcal{I}_{k-2j'},n\notin I}\abs{\frac{\partial^{k-j'} f(p)}{\partial y_I\partial y_n^{j'}}}^2.
\]
The $C^{k,j}_b$ are intermediate norms between $C^{k}_b$ and $C^{k-1}_b$. In particular,  $\abs{f}^2_{C^{k,0}_b(p)}=\abs{f}^2_{C^k_b(p)}$.
Also, when $k$ is even $\abs{f}_{C^{k,k/2}_b(p)}^2 = \abs{f}_{C^{k-1}_b(p)}^2 + \abs{\frac{\p^{k/2}f(p)}{\p y_n^{k/2}}}^2$ and when $k$ is odd $\abs{f}_{C^{k,(k-1)/2}_b(p)}^2 = \abs{f}_{C^{k-1}_b(p)}^2 + \sum_{\ell=1}^{n-1}\abs{\frac{\p^{(k+1)/2}f(p)}{\partial y_\ell\p y_n^{(k-1)/2}}}^2$.
In general, if $I$ is a multi-index such that $n\notin I$ and $j\geq 0$ is an integer then
\begin{equation}
\label{eq:special_norm_estimate}
  \abs{\frac{\partial^{\abs{I}+j} f(p)}{\partial y_I\partial y_n^j}}\leq\abs{f}_{C^{\abs{I}+2j,j}_b(p)}\leq\abs{f}_{C^{\abs{I}+2j}_b(p)}.
\end{equation}

The utility of this norm can be seen from the following lemma.
\begin{lem}
\label{lem:rho_delta}
  Let $\Omega\subset\mathbb{R}^n$ have $C^m$ boundary, $m\geq 2$.  Let $\rho$ be a $C^m$ defining function for $\Omega$ and let $h$ be the positive $C^{m-1}$ function defined in a neighborhood of $\partial\Omega$ by $\tilde\delta=h\rho$.  Then
  \[
    \sup_{p\in\partial\Omega}\frac{\abs{\rho}_{C^m_b(p)}}{\abs{\nabla\rho(p)}}<\infty.
  \]
  if and only if
  \[
    \|\tilde\delta\|_{C^m(\partial\Omega)}<\infty\textrm{ and }\sup_{p\in\partial\Omega}\frac{\abs{h}_{C^{m-2}_b(p)}}{h(p)}<\infty.
  \]
\end{lem}

\begin{rem}
  When $m=2$ the statement about $h$ is trivial, so the conditions on $\rho$ and $\tilde\delta$ are equivalent.  We will see in \eqref{eq:delta_C2} that something stronger is true in this case.
\end{rem}

\begin{proof}
Since $|\nabla\tilde\delta|^2=1$ on a neighborhood of $\partial\Omega$ (see \cite{KrPa81} and Theorem 4.8 (3) in \cite{Fed59}), for $I\in\mathcal{I}_k$ with $1\leq k\leq m-1$, we can differentiate this equality by
$\frac{\p^k}{\p x_I}$ to obtain
\[
  \sum_{j=1}^n\sum_{J\subseteq I}\frac{\partial}{\partial x_j}\left(\frac{\partial^{\abs{J}}\tilde\delta}{\partial x_J}\right)\frac{\partial}{\partial x_j}\left(\frac{\partial^{k-\abs{J}}\tilde\delta}{\partial x_{I\backslash J}}\right)=0.
\]
on $\partial\Omega$.  For fixed $p\in\partial\Omega$, choose coordinates $(y_1,\ldots,y_n)$ so that $p=0$ and $\nabla_y\tilde\delta(p)=(0,\ldots,0,1)$.  In these coordinates,
\begin{equation}
\label{eq:delta_normal_derivatives}
  \sum_{j=1}^n\sum_{J\neq\emptyset, J\subsetneq I}\frac{\partial}{\partial y_j}\left(\frac{\partial^{\abs{J}}\tilde\delta}{\partial y_J}\right)\frac{\partial}{\partial y_j}\left(\frac{\partial^{k-\abs{J}}\tilde\delta}{\partial y_{I\backslash J}}\right)(p)+2\frac{\partial}{\partial y_n}\left(\frac{\partial^{k}\tilde\delta}{\partial y_{I}}\right)(p)=0.
\end{equation}
From this, we conclude that
\begin{equation}
\label{eq:delta_normal_estimate}
  \abs{\frac{\partial}{\partial y_n}\left(\frac{\partial^{k}\tilde\delta}{\partial y_{I}}\right)(p)}\leq C_1\|\tilde\delta\|_{C^k(\partial\Omega)}^2.
\end{equation}
for some constant $C_1>0$ and for any $I\in\mathcal{I}_k$ with $1\leq k\leq m-1$.

Since $h$ is $C^{m-1}$, we may differentiate $\tilde\delta=h\rho$ in a neighborhood of $\partial\Omega$ by $\frac{\p^k}{\p x_I}$ for $I\in\mathcal{I}_k$ to obtain
\[
  \frac{\partial^k\tilde\delta}{\partial x_{I}}=\sum_{J\subseteq I}\frac{\partial^{\abs{J}}h}{\partial x_J}\frac{\partial^{k-\abs{J}}\rho}{\partial x_{I\backslash J}}.
\]
This can not be differentiated directly again if $k=m-1$ because $h$ is only $C^{m-1}$, but we may form a difference quotient at $p\in\partial\Omega$ and take the limit to obtain
\begin{equation}
\label{eq:k+1_derivatives}
  \frac{\partial^{k+1}\tilde\delta(p)}{\partial x_{I}}=\sum_{J\subsetneq I}\frac{\partial^{\abs{J}}h(p)}{\partial x_J}\frac{\partial^{k+1-\abs{J}}\rho(p)}{\partial x_{I\backslash J}}
\end{equation}
for any $I\in\mathcal{I}_{k+1}$, since $\rho(p)=0$.  In our special coordinates note that $\frac{\partial\rho}{\partial y_j}(p)=0$ if $j\neq n$, so if $n\notin I$ in these coordinates, all of the terms with first derivatives of $\rho$ will also vanish, leaving us with
\begin{equation}
\label{eq:delta_tangent_estimate}
  \abs{\frac{\partial^{k+1}\tilde\delta(p)}{\partial y_{I}}}\leq C_2\abs{h}_{C^{k-1}_b(p)}\abs{\rho}_{C^{k+1}_b(p)}
\end{equation}
for some constant $C_2>0$.

For $0\leq j\leq k'$ and $I\in\mathcal{I}_{k'-j}$ with $n\notin I$, we obtain from \eqref{eq:k+1_derivatives} the equation
\begin{multline*}
  \frac{\partial^{k'+1}\tilde\delta(p)}{\partial y_{I}\partial y_n^{j+1}}
  =\sum_{\ell=0}^{j}\sum_{J\subseteq I}\binom{j+1}{\ell}\frac{\partial^{\abs{J}+\ell}h(p)}{\partial y_J\partial y_n^\ell}\frac{\partial^{k'-\abs{J}+1-\ell}\rho(p)}{\partial y_{I\backslash J}\partial y_n^{j+1-\ell}}\\
  +\sum_{J\subset I,\abs{J}\leq k'-j-2}\frac{\partial^{\abs{J}+j+1}h(p)}{\partial y_J\partial y_n^{j+1}}\frac{\partial^{k'-j-\abs{J}}\rho(p)}{\partial y_{I\backslash J}}.
\end{multline*}
Subtracting the highest order terms in $h$ (with respect to the $C^k_b$ norm), we can use \eqref{eq:special_norm_estimate} to estimate the remainder by
\begin{multline*}
  \abs{\frac{\partial^{k'+1}\tilde\delta(p)}{\partial y_{I}\partial y_n^{j+1}}-(j+1)\frac{\partial^{k'}h(p)}{\partial y_I\partial y_n^j}\frac{\partial\rho(p)}{\partial y_n}
  -\sum_{J\subset I,\abs{J}=k'-j-2}\frac{\partial^{k'-1}h(p)}{\partial y_J\partial y_n^{j+1}}\frac{\partial^{2}\rho(p)}{\partial y_{I\backslash J}}}\\
  \leq C_3\abs{h}_{C^{k'+j-1}_b(p)}\abs{\rho}_{C^{k'+j+2}_b(p)}.
\end{multline*}
for some constant $C_3>0$ and integers $k'\geq j+2$.  If $j+2>k'\geq j$, we have simply
\[
  \abs{\frac{\partial^{k'+1}\tilde\delta(p)}{\partial y_{I}\partial y_n^{j+1}}-(j+1)\frac{\partial^{k'}h(p)}{\partial y_I\partial y_n^j}\frac{\partial\rho(p)}{\partial y_n}}
  \leq C_3\abs{h}_{C^{k'+j-1}_b(p)}\abs{\rho}_{C^{k'+j+2}_b(p)}.
\]
Suppose that $0\leq j\leq\frac{k-1}{2}$ and set $k'=k-j-1$ (so that $I\in \mathcal{I}_{k-2j-1}$).  Note that $\frac{\partial\rho}{\partial y_n}(p)=|\nabla\rho(p)|$ and $\frac{\partial\tilde\delta}{\partial y_n}(p)=1$, so by \eqref{eq:k+1_derivatives} with $k=0$ we have
\begin{equation}\label{eqn:h nabla rho =1}
h(p)|\nabla\rho(p)| = 1.
\end{equation}
Thus we have
\begin{multline}
\label{eq:h_estimate}
  (j+1)\abs{\frac{\partial^{k-j-1}h(p)}{\partial y_I\partial y_n^j}}h(p)^{-1}\\
  \leq \abs{\frac{\partial^{k-j}\tilde\delta(p)}{\partial y_{I}\partial y_n^{j+1}}}
  +\sum_{J\subset I,\abs{J}=k-2j-3}\abs{\frac{\partial^{k-j-2}h(p)}{\partial y_J\partial y_n^{j+1}}\frac{\partial^{2}\rho(p)}{\partial y_{I\backslash J}}}+C_3\abs{h}_{C^{k-2}_b(p)}\abs{\rho}_{C^{k+1}_b(p)}
\end{multline}
if $k\geq 2j+3$, and
\begin{equation}
\label{eq:h_estimate_0}
  (j+1)\abs{\frac{\partial^{k-j-1}h(p)}{\partial y_I\partial y_n^j}}h(p)^{-1}
  \leq \abs{\frac{\partial^{k-j}\tilde\delta(p)}{\partial y_{I}\partial y_n^{j+1}}}+C_3\abs{h}_{C^{k-2}_b(p)}\abs{\rho}_{C^{k+1}_b(p)}.
\end{equation}
if $k<2j+3$.

We now proceed by induction.  Assume $\sup_{p\in\partial\Omega}\frac{\abs{\rho}_{C^m_b(p)}}{\abs{\nabla\rho(p)}}<\infty$.  Suppose that for some $m-1\geq k\geq 1$, $\|\tilde\delta\|_{C^{k}(\partial\Omega)}<\infty$ and $\sup_{p\in\partial\Omega}\frac{\abs{h}_{C^{k-2}_b(p)}}{h(p)}<\infty$.  When $k=1$, this is clear since $\|\tilde\delta\|_{C^1(\partial\Omega)}=1$ and the condition on $h$ is vacuous.  Using $j=\lfloor\frac{k-1}{2}\rfloor$ with \eqref{eq:h_estimate_0} and the induction hypothesis we can show that $\sup_{p\in\partial\Omega}\frac{\abs{h}_{C^{k-1,\lfloor(k-1)/2\rfloor}_b(p)}}{h(p)}<\infty$.  Suppose that for some $0\leq j\leq\frac{k-3}{2}$ we know that $\sup_{p\in\partial\Omega}\frac{\abs{h}_{C^{k-1,j+1}_b(p)}}{h(p)}<\infty$.  Using \eqref{eq:h_estimate}, we know now that  $\sup_{p\in\partial\Omega}\frac{\abs{h}_{C^{k-1,j}_b(p)}}{h(p)}<\infty$ since $|h|_{C^{k-1,j}_b(p)} =|h|_{C^{k-1,j+1}_b(p)} + \sum_{\atopp{I\in \I_{k-2j-1}}{n\not\in I}} | \frac{\p^{k-j-1} h(p)}{\p y_I \p y_n^j} |$.
Proceeding by downward induction on $j$ we have $\sup_{p\in\partial\Omega}\frac{\abs{h}_{C^{k-1}_b(p)}}{h(p)}<\infty$.

Using \eqref{eq:delta_normal_estimate} and \eqref{eq:delta_tangent_estimate}, we conclude $\|\tilde\delta\|_{C^{k+1}(\partial\Omega)}<\infty$.  The result follows by induction on $k$.

For the converse, we simply subtract the highest degree term in $\rho$ from \eqref{eq:k+1_derivatives} to obtain for $0\leq j\leq k'+1$ and $I\in\mathcal{I}_{k'+1-j}$ with $n\notin I$
\[
  \abs{\frac{\partial^{k'+1}\tilde\delta(p)}{\partial y_{I}\partial y_n^{j}}-h(p)\frac{\partial^{k'+1}\rho(p)}{\partial y_{I}\partial y_n^{j}}}\leq C_4\abs{h}_{C_b^{k'+j-1}(p)}\abs{\rho}_{C_b^{k'+j}(p)},
\]
for some constant $C_4>0$.  If we set $k'=k-j$ then for any $0\leq j\leq \frac{k+1}{2}$ and $I\in\mathcal{I}_{k-2j+1}$ with $n\notin I$ we have
\[
  \abs{\nabla\rho(p)}^{-1}\abs{\frac{\partial^{k-j+1}\rho(p)}{\partial y_{I}\partial y_n^{j}}}\leq\abs{\frac{\partial^{k-j+1}\tilde\delta(p)}{\partial y_{I}\partial y_n^{j}}}+C_4\abs{h}_{C_b^{k-1}(p)}\abs{\rho}_{C_b^{k}(p)}.
\]
The result follows by induction on $k$.

\end{proof}

Although Lemma \ref{lem:rho_delta} may not apply to all $C^m$ defining functions, it will suffice to prove the main theorem.  However, the inductive procedure used to prove this lemma may also be used to construct a system of boundary invariants for any defining function.  We illustrate by considering the $m=2$ and $m=3$ cases.  By \eqref{eq:delta_normal_estimate}, it will suffice to consider derivatives in tangential directions.  Fix $1\leq j,k,\ell\leq n-1$.  In the special coordinates of Lemma \ref{lem:rho_delta} at $p$ we apply \eqref{eq:k+1_derivatives} repeatedly to obtain
\begin{gather*}
  1=h(p)|\nabla\rho(p)|,\qquad
  \frac{\partial^2\tilde\delta(p)}{\partial y_j\partial y_k}=h(p)\frac{\partial^2\rho(p)}{\partial y_j\partial y_k},\qquad
  \frac{\partial^2\tilde\delta(p)}{\partial y_j\partial y_n}=h(p)\frac{\partial^2\rho(p)}{\partial y_j\partial y_n}+\frac{\partial h(p)}{\partial y_j}|\nabla\rho(p)|,\\
  \frac{\partial^3\tilde\delta(p)}{\partial y_j\partial y_k\partial y_\ell}=h(p)\frac{\partial^3\rho(p)}{\partial y_j\partial y_k\partial y_\ell}
  +\frac{\partial h(p)}{\partial y_j}\frac{\partial^2\rho(p)}{\partial y_k\partial y_\ell}
  +\frac{\partial h(p)}{\partial y_k}\frac{\partial^2\rho(p)}{\partial y_j\partial y_\ell}+\frac{\partial h(p)}{\partial y_\ell}\frac{\partial^2\rho(p)}{\partial y_j\partial y_k}.
\end{gather*}
By \eqref{eq:delta_normal_derivatives}, $\frac{\partial^2\tilde\delta(p)}{\partial y_j\partial y_n}=0$, so we may use (\ref{eqn:h nabla rho =1}) and the previous equalities
to conclude
\begin{equation}
\label{eq:delta_C2}
  \frac{\partial^2\tilde\delta(p)}{\partial y_j\partial y_k}=|\nabla\rho|^{-1}\frac{\partial^2\rho(p)}{\partial y_j\partial y_k},
\end{equation}
and
\begin{multline}
\label{eq:delta_C3}
  \frac{\partial^3\tilde\delta(p)}{\partial y_j\partial y_k\partial y_\ell}=|\nabla\rho|^{-1}\frac{\partial^3\rho(p)}{\partial y_j\partial y_k\partial y_\ell}\\
  -|\nabla\rho|^{-2}\left(\frac{\partial^2\rho(p)}{\partial y_j\partial y_n}\frac{\partial^2\rho(p)}{\partial y_k\partial y_\ell}+\frac{\partial^2\rho(p)}{\partial y_k\partial y_n}\frac{\partial^2\rho(p)}{\partial y_j\partial y_\ell}
  +\frac{\partial^2\rho(p)}{\partial y_\ell\partial y_n}\frac{\partial^2\rho(p)}{\partial y_j\partial y_k}\right).
\end{multline}
Once we have completed the proof of the main theorem, we can derive necessary and sufficient conditions for the existence of uniformly $C^2$ (resp.\ $C^3$) defining functions by checking
the boundedness of \eqref{eq:delta_C2} (resp.\ \eqref{eq:delta_C2} and \eqref{eq:delta_C3}).  Higher order conditions can be derived as well, but these will be progressively more complicated.

To facilitate formulas without special coordinates, we define
\[
  T_p(\partial\Omega)=\set{t\in\mathbb{R}^n:\sum_{j=1}^n t_j\frac{\partial\tilde\delta}{\partial x_j}(p)=0}.
\]
We also use the notation $y = (y',y_n)$ for $y'\in\R^{n-1}$ and $y_n\in\R$.

\begin{lem}
  Let $\Omega\subset\mathbb{R}^n$ have a $C^2$ boundary.  Then for any $C^2$ defining function $\rho$ we have
  \begin{equation}
  \label{eq:C2_reach}
    \sup_{p\in\partial\Omega}\sup_{\atopp{t\in T_p(\partial\Omega)}{ \abs{t}=1}}|\nabla\rho|^{-1}\abs{\sum_{j,k=1}^n t^j\frac{\partial^2\rho}{\partial x_j\partial x_k}(p)t^k}<\infty
  \end{equation}
  if and only if $\partial\Omega$ has positive reach, and
  \begin{equation}
  \label{eq:C2_reach_computed}
    \rea(\partial\Omega)=\left(\sup_{p\in\partial\Omega}\sup_{\atopp{t\in T_p(\partial\Omega)}{ \abs{t}=1}}|\nabla\rho|^{-1}\abs{\sum_{j,k=1}^n t^j\frac{\partial^2\rho}{\partial x_j\partial x_k}(p)t^k}\right)^{-1}.
  \end{equation}
\end{lem}

\begin{proof}
  For $p\in\partial\Omega$, choose local coordinates $(y_1,\ldots,y_n)$ so that $p=0$ and $\nabla\tilde\delta(p)=(0',1)$.  Suppose that for some $r>0$, $B((0',r),r)\subset\Omega^c$ and
$B((0',-r),r)\subset\Omega$.  Then for $y\in\partial\Omega$, $|y-(0',\pm r)|^2 \geq r^2$, so $\abs{y}^2\mp 2 y_n r\geq 0$.  Hence $\frac{\abs{y}^2}{2r}\geq \abs{y_n}$.
By Theorem 4.18 in \cite{Fed59}, this can be accomplished at every $p\in\partial\Omega$ if and only if $\rea(\partial\Omega)\geq r$.  Since the boundary is $C^2$, this is possible at each point if and only if
  \[
    \sum_{j,k=1}^n \abs{t_j\frac{\partial^2\tilde\delta(p)}{\partial x_j\partial x_k}t_k}\leq\frac{\abs{t}^2}{r}
  \]
  on $\partial\Omega$ for any vector $t\in T_p(\partial\Omega)$.  By \eqref{eq:delta_C2}, this is equivalent to
  \[
    |\nabla\rho|^{-1}\abs{\sum_{j,k=1}^n t_j\frac{\partial^2\rho(p)}{\partial x_j\partial x_k}t_k}\leq\frac{\abs{t}^2}{r}
  \]
  for all $p\in\partial\Omega$ and $t\in T_p(\partial\Omega)$.  If we take the supremum over all possible $r>0$ satisfying these inequalities, the result follows.
\end{proof}

\begin{lem}
\label{lem:delta_neighborhood}
  Let $\Omega\subset\mathbb{R}^n$ have a $C^{m}$ boundary for some $m\geq 2$ and suppose that the signed distance function for $\Omega$ satisfies $\|\tilde\delta\|_{C^m(\partial\Omega)}<\infty$.  Then for any $0<\epsilon<\rea(\partial\Omega)$ the signed distance function satisfies $\|\tilde\delta\|_{C^m(U)}<\infty$ on $U_\epsilon=\set{x\in\mathbb{R}^n:\delta(x)<\epsilon}$.
\end{lem}

\begin{proof}
  By the previous lemma, $\partial\Omega$ has positive reach, so $\tilde\delta$ is a $C^{m}$ function on a neighborhood $U'\supset \p\Omega$ \cite{KrPa81}.  Note that the result of Krantz and Parks is essentially local, so it is possible that $d(\partial U',\partial\Omega)=0$ if $\partial\Omega$ is not compact.  Set
  \[
    U=\set{x\in\mathbb{R}^n:\delta(x)<\rea(\partial\Omega)}.
  \]
  By Theorem 4.8 (3) and (5) in \cite{Fed59}, for any $x\in U$ we have $\nabla\tilde\delta(x)=\nabla\tilde\delta(\pi(x))$, where $\pi(x)=x-\tilde\delta(x)\nabla\tilde\delta(x)\in\partial\Omega$ is the unique boundary point nearest to $x$.  This is differentiable, and solving the derivative for $\nabla^2\tilde\delta$  gives us
  \begin{equation}
  \label{eq:delta_second_derivatives_neighborhood}
    \frac{\partial^2\tilde\delta(x)}{\partial x_j\partial x_\ell}=\sum_{\ell'=1}^n\frac{\partial^2\tilde\delta(\pi(x))}{\partial x_j\partial x_{\ell'}}\left(Id+\tilde\delta(x)\nabla^2\tilde\delta(\pi(x))\right)^{-1}_{\ell'\ell}
  \end{equation}
  for $x\in U$, where $Id$ is the identity matrix (see \cite{Wei75} and \cite{HeMc10}; see also \eqref{eq:hessian_eigenvalues} below).  Note that $Id+\tilde\delta(x)\nabla^2\tilde\delta(x)$ is invertible on $U$ by \eqref{eq:delta_C2} and \eqref{eq:C2_reach_computed}.  This formula shows that $\tilde\delta$ is $C^2$ on $U$ (we already know that $\tilde\delta$ is $C^2$ near $\p\Omega$ and $\pi(x)\in\partial\Omega$).
  Since this formula relates derivatives away from $\partial\Omega$ to derivatives on $\partial\Omega$ (which exist since $\partial\Omega\subset U'$), we may continue to differentiate and use induction to show that $\tilde\delta$ is $C^m$ on $U$.

  Fix $p\in\partial\Omega$ and choose new coordinates $(y_1,\ldots,y_n)$ so that $p=0$, $\nabla\tilde\delta(p)=(0',1)$, and $\nabla^2\tilde\delta(p)$ is diagonalized with eigenvalues $\kappa_1,\ldots,\kappa_n$.  By Theorem 4.8 (3) in \cite{Fed59}, when $y'=0'$ and $\abs{y_n}<\rea(\partial\Omega)$, we have $\tilde\delta(y)=y_n$ and $\nabla\tilde\delta(y)=(0',1)$.  Differentiating $|\nabla\tilde\delta|^2=1$ once demonstrates that $\kappa_n=0$.  For $m\geq 3$, differentiating $|\nabla\tilde\delta|^2=1$ twice yields
  \[
    2\sum_{\ell=1}^n\left(\frac{\partial\tilde\delta}{\partial x_\ell}\frac{\partial^3\tilde\delta}{\partial x_\ell\partial x_j\partial x_k}
    +\frac{\partial^2\tilde\delta}{\partial x_\ell\partial x_j}\frac{\partial^2\tilde\delta}{\partial x_\ell\partial x_k}\right)=0
  \]
  on $U$.  From \eqref{eq:delta_second_derivatives_neighborhood}, we can see that eigenvectors of $\nabla^2\tilde\delta$ are preserved along the normal direction. Rewriting the above equation in our $y$-coordinates, when $j=k$ we have $2\left(\frac{\partial\kappa_j}{\partial y_n}+\kappa_j^2\right)(y)=0$ on $U$ when $y'=0'$.  The unique solution to this equation is given by
  \begin{equation}
  \label{eq:hessian_eigenvalues}
    \kappa_j(y)=\frac{\kappa_j(0)}{1+y_n\kappa_j(0)}
  \end{equation}
  (see also Lemma 14.17 in \cite{GiTr01}, but with the opposite sign convention).  Since $\rea(\partial\Omega)\leq\abs{\kappa_j(0)}^{-1}$ (see \eqref{eq:delta_C2} and \eqref{eq:C2_reach_computed}) for all $1\leq j\leq n-1$ with $\kappa_j\neq 0$, $\kappa_j$ will be uniformly bounded on $U_\epsilon$.

  For $3\leq k\leq m-1$, let $I\in\mathcal{I}_k$.  Then differentiating $|\nabla\tilde\delta|^2=1$ gives us
  \[
    \sum_{j=1}^n\sum_{J\subseteq I}\frac{\partial}{\partial x_j}\left(\frac{\partial^{\abs{J}}\tilde\delta}{\partial x_J}\right)\frac{\partial}{\partial x_j}\left(\frac{\partial^{k-\abs{J}}\tilde\delta}{\partial x_{I\backslash J}}\right)=0.
  \]
  on $U$.  In our diagonalized coordinates, we can evaluate terms involving only first or second derivatives separately to obtain
  \[
    2\frac{\partial^{k+1}\tilde\delta}{\partial y_n\partial y_{I}}+2\sum_{j=1}^k\kappa_{I_j}\frac{\partial^k\tilde\delta}{\partial y_{I}}
    +\sum_{j=1}^n\sum_{J\subset I,2\leq\abs{J}\leq k-2}\frac{\partial}{\partial y_j}\left(\frac{\partial^{\abs{J}}\tilde\delta}{\partial y_J}\right)\frac{\partial}{\partial y_j}\left(\frac{\partial^{k-\abs{J}}\tilde\delta}{\partial  y_{I\backslash J}}\right)=0
  \]
  on $U$ when $y'=0'$.  If we set
  \[
    \mu_{I}(y_n)=\prod_{j=1}^k(1+y_n\kappa_{I_j}(0)),
  \]
  then $\mu_{I}(y_n)$ solves the initial value problem
  \[
    \frac{\partial\mu_{I}}{\partial y_n}(y_n)=\mu_{I}(y_n)\sum_{j=1}^k\kappa_{I_j}(0',y_n)\text{ and }\mu_{I}(0)=1,
  \]
  so
  \[
    2\frac{\partial}{\partial y_n}\left(\mu_{I}\frac{\partial^{k}\tilde\delta}{\partial y_{I}}\right)+\mu_{I}\sum_{j=1}^n\sum_{J\subset I,2\leq\abs{J}\leq k-2}\frac{\partial}{\partial y_j}\left(\frac{\partial^{\abs{J}}\tilde\delta}{\partial y_J}\right)\frac{\partial}{\partial y_j}\left(\frac{\partial^{k-\abs{J}}\tilde\delta}{\partial y_{I\backslash J}}\right)=0
  \]
  on $U$ when $y'=0'$.  Hence, we may integrate to obtain
  \begin{multline}
  \label{eq:delta_derivatives}
    \frac{\partial^{k}\tilde\delta}{\partial y_{I}}(0',y_n)=\frac{1}{\mu_{I}(y_n)}\frac{\partial^{k}\tilde\delta}{\partial y_{I}}(0)\\
    -\frac{1}{2\mu_{I}(y_n)}\int_0^{y_n}\mu_{I}(t)\sum_{j=1}^n\sum_{J\subset I,2\leq\abs{J}\leq k-2}\frac{\partial}{\partial y_j}\left(\frac{\partial^{\abs{J}}\tilde\delta}{\partial y_J}\right)\frac{\partial}{\partial y_j}\left(\frac{\partial^{k-\abs{J}}\tilde\delta}{\partial y_{I\backslash J}}\right)(0',t)dt.
  \end{multline}
  Since $\mu_{I}(y_n)$ is uniformly bounded below on $U_\epsilon$ and the terms in the integral are differentiated at most $k-1$ times, we may use induction on $k$ to obtain uniform bounds on $\frac{\partial^{k}\tilde\delta}{\partial y_{I}}$ on $U_\epsilon$ for all $I$ with $3\leq k\leq m-1$.

  Now, we wish to differentiate our formulas for $k=m-1$ to show that they also hold for $k=m$.  By differentiating \eqref{eq:delta_second_derivatives_neighborhood}, we can obtain formulas for the first $m$ derivatives of $\tilde\delta$ on $U$ in terms of derivatives restricted to $\partial\Omega$.  By formal manipulations, these must be equivalent to those obtained in \eqref{eq:delta_derivatives}, and hence the first $m$ derivatives remain uniformly bounded on $U_\epsilon$.
  \end{proof}

\section{Examples in Projective Space}
\label{sec:projective_space}

A large class of examples of domains with uniformly $C^m$ defining functions can be found by considering $\mathbb{R}^n\subset\mathbb{RP}^n$.
Recall that $\mathbb{RP}^n=(\mathbb{R}^n\backslash\set{0})/\sim$ under the equivalence relation $x\sim y$ if $x=\lambda y$ for $\lambda\in\mathbb{R}\backslash\set{0}$.
If we denote coordinates on $\mathbb{RP}^n$ by $[x_1:\ldots:x_{n+1}]$,  the canonical embedding of $\mathbb{R}^n$ in $\mathbb{RP}^n$ is given by $(x_1,\ldots,x_n)\mapsto[x_1:\ldots:x_n:1]$.
Every unbounded domain $\Omega$ in $\mathbb{R}^n$ can be extended to a bounded domain $\tilde\Omega$ in $\mathbb{RP}^n$ with respect to this embedding. Conversely, from a domain
$\tilde\Omega\subset\R\P^n$, we can canonically produce a (possibly unbounded) domain $\Omega\subset\R^n$ under the mapping $[x_1:\ldots:x_n:1] \mapsto (x_1,\dots,x_n)$.
\begin{cor}
  Let $\tilde\Omega\subset\mathbb{RP}^n$ be a $C^m$ domain.  Then the domain $\Omega\subset\mathbb{R}^n$ obtained by pulling back along the canonical embedding has a uniformly $C^m$ defining function.
\end{cor}
\begin{proof}
  For $S^n\subset\mathbb{R}^{n+1}$, we may define $\tilde\Omega$ by a $C^m$ defining function $\tilde\rho:S^n\rightarrow\mathbb{R}$ such that $\tilde\rho(-x)=\tilde\rho(x)$.  Extend $\tilde\rho$ to all of $\mathbb{R}^{n+1}\backslash\set{0}$ by $\tilde\rho(x)=\tilde\rho\left(\frac{x}{\abs{x}}\right)$.  Since we have $\tilde\rho(x)=\tilde\rho(\lambda x)$ for any $\lambda\in\mathbb{R}\backslash\set{0}$, we also obtain $\nabla^k\tilde\rho(x)=\lambda^k(\nabla^k\tilde\rho)(\lambda x)$.  If we assume that $\abs{\nabla^k\tilde\rho}<C_k$ and $\abs{\nabla\tilde\rho}>C_0$ on $\partial\tilde\Omega\cap S^n$ for any integer $1\leq k\leq m$ and some constants $C_0,C_k>0$, then we have in general $\abs{\nabla^k\tilde\rho(x)}<C_k\abs{x}^{-k}$ and $\abs{\nabla\tilde\rho(x)}>C_0\abs{x}^{-1}$ whenever $\tilde\rho(x)=0$.

  A defining function $\rho$ for $\Omega\subset\mathbb{R}^n$ can now be obtained by considering $\rho(x_1,\ldots,x_n)=\tilde\rho(x_1,\ldots,x_n,1)$.  Thus $\abs{\nabla^k\rho(x)}<\frac{C_k}{(1+\abs{x}^2)^{k/2}}$ and $\abs{\nabla\rho(x)}>\frac{C_0}{(1+\abs{x}^2)^{1/2}}$ on $\partial\Omega$, so
  \[
    \frac{\abs{\nabla^k\rho}}{|\nabla\rho|}(x)<\frac{C_k}{C_0(1+\abs{x}^2)^{(k-1)/2}}
  \]
  on $\partial\Omega$ for all $1\leq k\leq m$.  By our main theorem, this implies that $\Omega$ has a uniformly $C^m$ defining function.
\end{proof}
Note that this proof can still be used if $\tilde\rho$ is $C^m$ when $x_{n+1}\neq 0$ and $\abs{\nabla^k\tilde\rho(x)}<C_k x_{n+1}^{1-k}$ for $x\in S^n$ with $x_{n+1}\neq 0$, so a uniformly $C^m$ defining function in $\mathbb{R}^n$ covers a much larger class of examples than those given by $C^m$ domains in $\mathbb{RP}^n$.

For example, consider the domain $\Omega_1\subset\mathbb{R}^2$ defined by
\[
  \Omega_1=\set{y<x^{-1}\sin x,x\neq 0}\cup\set{y<1,x=0}.
\]
Then $\Omega_1$ is a $C^\infty$ domain.  Let $\rho_1(x,y)=y-x^{-1}\sin x$ when $x\neq 0$ and $\rho_1(0,y)=y-1$.  By considering the Maclaurin series of $\sin x$ we can see that $\rho_1$ is real-analytic (hence smooth) in a neighborhood of the set where $x=0$.  When $x\neq 0$, we have $\nabla\rho_1=(x^{-2}\sin x-x^{-1}\cos x,1)$, so $\abs{\nabla\rho_1}$ is uniformly bounded above and away from zero.  Differentiating $m$ times, we have $\abs{\frac{\partial^m\rho_1}{\partial x^m}}\leq O(x^{-1})$, so this is also uniformly bounded.  Hence $\rho_1$ is a uniformly $C^m$ defining function for any integer $m$.  In $\mathbb{RP}^2$, this defining function can be written $\rho_1([x:y:z])=\frac{y}{z}-\frac{z}{x}\sin\left(\frac{x}{z}\right)$.  On the coordinate patch where $x\neq 0$, this can be written $\rho_1(y,z)=\frac{y}{z}-z\sin(1/z)$.  To normalize this near $z=0$, we use $\tilde\rho_1(y,z)=y-z^2\sin(1/z)$.  Note that for fixed $y$ this is the classic example of a function which is differentiable at $z=0$ but not $C^1$ in a neighborhood of $z=0$.  We conclude that $\tilde\Omega_1\subset\mathbb{RP}^2$ is not a $C^1$ domain.

On the other hand, consider $\Omega_2\subset\mathbb{R}^2$ defined by
\[
  \Omega_2=\set{y<x^{-2}\sin x^2,x\neq 0}\cup\set{y<1,x=0}.
\]
Let $\rho_2(x,y)=y-x^{-2}\sin x^2$ when $x\neq 0$ and $\rho_2(0,y)=y-1$.  Again, the Maclaurin series will show that all derivatives are uniformly bounded near $x=0$, so we focus on $x\neq 0$.  Since $\nabla\rho_2=(2x^{-3}\sin x^2-2x^{-1}\cos x^2,1)$, we define $D_1=\frac{1}{\abs{\nabla\rho_2}}\frac{\partial}{\partial x}-\frac{2x^{-3}\sin x^2-2x^{-1}\cos x^2}{\abs{\nabla\rho_2}}\frac{\partial}{\partial y}$ and $D_2=\frac{2x^{-3}\sin x^2-2x^{-1}\cos x^2}{\abs{\nabla\rho_2}}\frac{\partial}{\partial x}+\frac{1}{\abs{\nabla\rho_2}}\frac{\partial}{\partial y}$ to represent the tangent and normal directions on the boundary.  Since $\frac{\partial^2\rho_2}{\partial x^2}=4\sin x^2+O(x^{-2})$, $\rho_2$ is a uniformly $C^2$ defining function for $\Omega$, and hence $\partial\Omega$ has positive reach.  However, $\frac{\partial^3\rho_2}{\partial x^3}=8x\cos x^2+O(x^{-1})$.  If we fix $p = (p_x,p_y)\in\partial\Omega$ with $p_x\neq 0$, then \eqref{eq:delta_C3} tells us that
\[
\begin{split}
  (D_1|_p)^3\tilde\delta(p)&=\abs{\nabla\rho_2}^{-1}(D_1|_p)^3\rho_2(p)-3\abs{\nabla\rho_2}^{-2}((D_1|_p)(D_2|_p)\rho_2(p))((D_1|_p)^2\rho_2(p))\\
  &=8x\cos x^2+O(x^{-1}).
\end{split}
\]
This is not uniformly bounded, so there does not exist a uniformly $C^3$ defining function for $\partial\Omega$.  Hence, positive reach does not suffice to extend $C^m$ defining functions as uniformly $C^m$ defining functions.

Finally, let $h(x,y)=e^{x^2}$.  Then $\sup\frac{\abs{h}_{C_b^1}}{h}=\infty$ with respect to either of the previous two examples (since $\frac{\partial}{\partial x}$ is asymptotically the tangential direction in these examples).  By Lemma \ref{lem:rho_delta}, the defining function $\rho^h_1(x,y)=(y-x^{-1}\sin x)h(x,y)$ fails to satisfy $\sup\frac{\abs{\rho^h_1}_{C_b^3}}{\abs{\nabla\rho^h_1}}<\infty$ even though this defines a domain with a uniformly $C^3$ defining function.  Hence, not all defining functions need satisfy the conditions of Lemma \ref{lem:rho_delta}.  Turning to our other example, $\rho_2^h(x,y)=(y-x^{-2}\sin x^2)h(x,y)$ still satisfies $\sup\frac{\abs{\rho_2^h}_{C_b^2}}{\abs{\nabla\rho_2^h}}<\infty$ even though $\pnorm{\rho_2^h}_{C^2(\partial\Omega)}=\infty$. Thus, it is helpful to consider the special
$C_b^k$ norm  in place of the standard extrinsic $C^k$ norm.

%%%%%%%%%%%%%%%%%%%%%
%
%			SECTION: APPLICATIONS	
%
%%%%%%%%%%%%%%%%%%%%%
\section{Applications of uniformly $C^m$ defining functions}\label{sec:applications}
In \cite{HaRa13s}, we define weighted Sobolev spaces on the boundaries of unbounded domains.  From the standpoint of the present paper, the weight function is irrelevant.  However, it seems difficult to obtain elliptic regularity results without a weight, so for the sake of defining a meaningful space of functions we will include the weight.  The weight functions that we use satisfy a number of technical hypotheses (similar to those
in \cite{HaHe07,Gan10,Gan11}) all satisfied by $\vp_t(x) = t|x|^2$, $t\in\R\setminus\{0\}$. Let $\Omega\subset\R^n$ have a $C^m$  boundary, $m\geq 2$, that admits
a uniformly $C^m$ defining function. We define weighted Sobolev spaces both on the boundary and near the boundary.

Suppose $\Omega\subset\mathbb{R}^n$ admits a uniformly $C^2$ defining function.  By Theorem \ref{thm:main thm}, $\partial\Omega$ has positive reach, so for $\frac{1}{2}\rea(\partial\Omega)>\ep>0$ we set
\[
\Omega_\ep = \{x\in\Omega:\delta(x)<\ep\}.
\]
Since $\|\tilde\delta\|_{C^2(\Omega_{2\epsilon})}<\infty$, there exists a radius $\frac{1}{4}\rea(\partial\Omega)>r>0$ such that whenever $B(p,r)\cap\Omega_\epsilon\neq\emptyset$, there exist coordinates on $B(p,r)$ such that the level curves of $\tilde\delta$ can be written as a graph.  Hence, there exists an orthonormal basis
$L_1,\dots,L_{n-1}$ of the tangent space to the level curves of $\tilde\delta$ on $B(p,r)$. We also let  $L_n = \nu$ be the unit outward normal to the level curves of $\tilde\delta$. For $1\leq j \leq n$, set
\[
T_j = L_j - L_j(\vp_t).
\]
We call a first order differential operator $T$ \emph{tangential} if the first order component of $T$ is tangential.  Note that we use $T_j$ instead of $L_j$ for technical reasons involving integration by parts in weighted norms, but these are not relevant for the present paper.

Let $\set{p_j}$ be an enumeration of all points in $\mathbb{R}^n$ whose coordinates are integral multiples of $\frac{r}{\sqrt{n}}$.  Then $\set{\overline{B(p_j,r/2)}}$ is a locally finite cover of $\mathbb{R}^n$, with a uniform upper bound on the number of sets covering each point.  If $\chi\in C^\infty_0(B(0,r))$ satisfies $\chi=1$ on $B(0,r/2)$ and $1\geq\chi\geq 0$, we can construct a partition of unity subordinate to $\set{B(p_j,r)}$ by using $\chi_j(x)=\chi(x-p_j)/\left(\sum_k\chi(x-p_k)\right)$.  Because there is a uniform upper bound on the number of nonzero terms in the denominator, we have a uniform bound on $\|\chi_j\|_{C^m}$ for any $m\geq 0$.

Let $\set{U_j}$ be a restriction of this cover to include only those sets covering $\Omega_\epsilon$, with the corresponding modification to $\chi_j$.  For any distribution $v$ on $\Omega_\epsilon$, we set $v_j = v \chi_j$, so $v = \sum_{j=1}^\infty v_j$.
If $\Omega$ admits a uniformly $C^m$ defining function, $m\geq 2$, we define the weighted Sobolev space $W^{k,p}(\Omega_\ep,\vp_t,\nabla\vp_t)$, $0\leq k\leq m$, as the space of distributions $v$ on $\Omega_\ep$
whose partial derivatives up to order $k$ agree with functions and have the
following norm finite:
\[
\norm v \norm_{W^{k,p}(\Omega_\ep,\vp_t,\nabla\vp_t)}^p = \sum_{j=1}^\infty \sum_{|\alpha|\leq k} \| T^\alpha v_j \|_{L^p(\Omega_\ep,\vp_t)}^p
\]
where $T_j = L_j - L_j(\vp_t)$ is well-defined on $U_j$ and the composition $T^\alpha = T_{\alpha_1}\cdots T_{\alpha_{|\alpha|}}$.

For the boundary Sobolev space, set
\begin{multline*}
W^{k,p}(\partial\Omega,\vp_t,\nabla\vp_t) \\ = \{ f \in L^p(\partial\Omega,\vp_t): T^\alpha f \in L^p(\partial\Omega,\vp_t),\ |\alpha|\leq k \text{ and }T_{\alpha_j} \text{ is tangential for } 1\leq j \leq |\alpha|\}.
\end{multline*}

Choosing uniform neighborhoods with good local coordinates only makes sense on domains with positive reach, and compositions of derivatives would be extremely difficult to control without a uniformly
$C^m$ defining function. When $p=2$, we define fractional Sobolev spaces via interpolation and can prove many of the standard Sobolev space results.

We also provide an example from several complex variables.  The following theorem is well known in the bounded case (see for example Theorem 3.4.4 in \cite{ChSh01}).
\begin{thm}
  Let $\Omega\subset\mathbb{C}^n$ be a domain with a $C^2$ defining function $r$ and a constant $C>0$ satisfying
  \begin{equation}
  \label{eq:strict_pseudoconvexity}
    \sum_{j,k=1}^n \frac{\partial^2r}{\partial z_j\partial\bar{z}_k}t_j\bar{t}_k\geq C\abs{\nabla r}\sum_{j=1}^n\abs{t_j}^2
  \end{equation}
  on $\partial\Omega$ for $t\in T_p(\partial\Omega)$.
  \begin{enumerate}
    \item If $\Omega$ admits a uniformly $C^2$ defining function, then there exists a defining function which is strictly plurisubharmonic on $\partial\Omega$.

    \item If $\Omega$ admits a uniformly $C^3$ defining function, then there exists a defining function which is plurisubharmonic on $\Omega$ and strictly plurisubharmonic on $\{z\in\Omega : \delta(z) < \eps\}$ for some $\eps>0$.
  \end{enumerate}
\end{thm}
\begin{rem}
  The assumption \eqref{eq:strict_pseudoconvexity} implies that $\Omega$ is strictly pseudoconvex, but the uniform lower bound on the Levi-form is not true for all strictly pseudoconvex domains in the unbounded case.  For an example that satisfies our condition, consider the tube in $\mathbb{C}^n$ defined by the defining function $r(z)=\abs{z'}^2+(\im z_n)^2-1$.
\end{rem}
\begin{rem}
  The second statement is not sharp in the bounded case, where \eqref{eq:strict_pseudoconvexity} alone (without the $C^3$ assumption) will guarantee the existence of a strictly plurisubharmonic defining function.  We require $C^3$ to govern the decay rate of \eqref{eq:strict_pseudoconvexity} off of $\partial\Omega$, and our resulting function is merely plurisubharmonic because we can not use $\abs{z}^2$ to obtain strict plurisubharmonicity in the interior (it is no longer a bounded function).
\end{rem}
\begin{proof}
  Since \eqref{eq:strict_pseudoconvexity} is independent of the choice of defining function, it will be satisfied by the signed distance function.  For $\lambda>0$ to be determined later, define
  \[
    \rho(z)=e^{\lambda\tilde\delta(z)}-1
  \]
  for $z$ in a small neighborhood of $\partial\Omega$.  For $v:\Omega_\epsilon\rightarrow\mathbb{C}^n$, we may
  decompose $v=\tau+\nu$, where $\sum_{j=1}^n\frac{\partial\tilde\delta}{\partial z_j}\tau_j=0$ and $\nu$ is a scalar multiple of
  $\left(\frac{\partial\tilde\delta}{\partial \bar{z}_1},\ldots,\frac{\partial\tilde\delta}{\partial \bar{z}_n}\right)$.
  Then since $\left|\sum_{j=1}^n\frac{\partial\tilde\delta}{\partial z_j}v_j\right|=\frac{1}{2}\sqrt{\sum_{j=1}^n\abs{\nu_j}^2}$ we have
  \[
    \sum_{j,k=1}^n \frac{\partial^2\rho}{\partial z_j\partial\bar{z}_k}v_j \bar{v}_k=\lambda e^{\lambda\tilde\delta}\sum_{j,k=1}^n \frac{\partial^2\tilde\delta}{\partial z_j\partial\bar{z}_k}v_j \bar{v}_k+\lambda^2 e^{\lambda\tilde\delta}\frac{1}{4}\sum_{j=1}^n\abs{\nu_j}^2.
  \]
  Since $\tilde\delta$ satisfies \eqref{eq:strict_pseudoconvexity} on $\partial\Omega$, we have
  \begin{multline*}
    \sum_{j,k=1}^n \frac{\partial^2\rho}{\partial z_j\partial\bar{z}_k}v_j \bar{v}_k\geq\lambda C\sum_{j=1}^n\abs{\tau_j}^2+\lambda \sum_{j,k=1}^n 2\re\left(\frac{\partial^2\tilde\delta}{\partial z_j\partial\bar{z}_k}\tau_j \bar{\nu}_k\right)\\
    +\lambda \sum_{j,k=1}^n \frac{\partial^2\tilde\delta}{\partial z_j\partial\bar{z}_k}\nu_j \bar{\nu}_k+\lambda^2 \frac{1}{4}\sum_{j=1}^n\abs{\nu_j}^2
  \end{multline*}
  on $\partial\Omega$.  Since $\tilde\delta$ is uniformly $C^2$, there exists a constant $C_2>0$ such that $\abs{\frac{\partial^2\tilde\delta}{\partial z_j\partial\bar{z}_k}}\leq C_2$ on $\partial\Omega$.  Hence, the Cauchy-Schwarz inequality gives us
  \[
    \sum_{j,k=1}^n \frac{\partial^2\rho}{\partial z_j\partial\bar{z}_k}v_j \bar{v}_k\geq\lambda C\abs{\tau}^2-2\lambda C_2\abs{\tau}\abs{\nu}
    -\lambda C_2\abs{\nu}^2+\lambda^2 \frac{1}{4}\abs{\nu}^2.
  \]
  This is strictly positive provided that $C\left(\frac{1}{4}\lambda-C_2\right)> 4C_2^2$.  Hence, we may choose $\lambda$ sufficiently large so that $\rho$ is strictly plurisubharmonic on $\partial\Omega$.

  If we assume that $\tilde\delta$ is uniformly $C^3$, then from \eqref{eq:strict_pseudoconvexity} there exists some uniform neighborhood $U$ of $\partial\Omega$ on which
  \[
    \sum_{j,k=1}^n \frac{\partial^2\tilde\delta}{\partial z_j\partial\bar{z}_k}\tau_j\bar{\tau}_k\geq \frac{1}{2}C\sum_{j=1}^n\abs{\tau_j}^2.
  \]
  We may assume $\abs{\frac{\partial^2\tilde\delta}{\partial z_j\partial\bar{z}_k}}\leq C_2$ on $U$, so that
  \[
    \sum_{j,k=1}^n \frac{\partial^2\rho}{\partial z_j\partial\bar{z}_k}v_j \bar{v}_k\geq\lambda e^{\lambda\tilde\delta}\frac{1}{2}C\abs{\tau}^2+2\lambda e^{\lambda\tilde\delta}C_2\abs{\tau}\abs{\nu}
    +\lambda e^{\lambda\tilde\delta}C_2\abs{\nu}^2+\lambda^2 e^{\lambda\tilde\delta}\frac{1}{4}\abs{\nu}^2.
  \]
  This is positive provided that $\frac{1}{2}C\left(\frac{1}{4}\lambda-C_2\right)\geq 4C_2^2$, so we may again choose $\lambda$ sufficiently large so that $\rho$ is plurisubharmonic on $\partial\Omega$.

  To extend $\rho$ to all of $\Omega$, let $A=\sup_{\Omega\backslash U}\tilde\delta$.  Since we were able to choose a uniform neighborhood $U$, $A<0$.  $\hat\rho=\max\set{\rho,A}$ will be a Lipschitz plurisubharmonic defining function for $\Omega$, and a smooth convex approximation to $\max$ can be used to obtain a smooth plurisubharmonic defining function for $\Omega$.
\end{proof}

\bibliographystyle{alpha}
\bibliography{mybib11-8-11}

\end{document}